\documentclass[12pt]{article}
\usepackage[left=1in,right=1in,top=1in,bottom=1in]{geometry}

\usepackage{amsthm,amssymb}

\def\COH{\mathsf{COH}}
\def\RCA{\mathsf{RCA}_0}
\def\RCAstar{\mathsf{RCA}^*_0}
\def\WKL{\mathsf{WKL}}

\def\Deltacomp{\Delta^0_1\mbox{-}\mathrm{CA}}

\def\CD{\mathcal{D}}
\def\CI{\mathcal{I}}
\def\CM{\mathcal{M}}
\def\CN{\mathcal{N}}
\def\CP{\mathcal{P}}
\def\CS{\mathcal{S}}

\def\ba{\mathbf{a}}
\def\bb{\mathbf{b}}
\def\bc{\mathbf{c}}
\def\bp{\mathbf{p}}
\def\bx{\mathbf{x}}
\def\by{\mathbf{y}}
\def\bzero{\mathbf{0}}

\def\halts{\downarrow}
\def\nohalts{\uparrow}
\def\rest{\upharpoonright}
\def\leqT{\leq_T}
\def\proves{\vdash}

\def\ISz{\mathrm{I}\Sigma^0_0}
\def\BSo{\mathrm{B}\Sigma^0_1}
\def\ISo{\mathrm{I}\Sigma^0_1}
\def\BSt{\mathrm{B}\Sigma^0_2}
\def\ISt{\mathrm{I}\Sigma^0_2}
\def\BSn{\mathrm{B}\Sigma^0_n}
\def\ISn{\mathrm{I}\Sigma^0_n}
\def\BSnp{\mathrm{B}\Sigma^0_{n+1}}

\def\ISnp{\mathrm{I}\Sigma^0_{n+1}}

\def\IDo{\mathrm{I}\Delta^0_1}
\def\IDt{\mathrm{I}\Delta^0_2}
\def\IDn{\mathrm{I}\Delta^0_n}
\def\IDnp{\mathrm{I}\Delta^0_{n+1}}

\newtheorem{theorem}{Theorem}
\newtheorem*{theorem*}{Theorem}
\newtheorem{axiom}{Axiom}[section]

\newtheorem{corollary}[axiom]{Corollary}

\newtheorem{lemma}[axiom]{Lemma}

\newtheorem{proposition}[axiom]{Proposition}
\newtheorem*{proposition*}{Proposition}
\newtheorem{question}[axiom]{Question}

\theoremstyle{definition}

\newtheorem*{remark}{Remark}

\begin{document}

\title{Conservation theorems for the Cohesiveness Principle}

\author{David R.\ Belanger \\ National University of Singapore and Ghent University}
\date{This version 2022}

\maketitle

\abstract{We prove that the Cohesiveness Principle (COH) is $\Pi^1_1$ conservative over $\RCA + \ISn$ and over $\RCA + \BSn$ for all $n \geq 2$ by recursion-theoretic means. We first characterize $\COH$ over $\RCA + \BSt$ as a `jumped' version of Weak K\"{o}nig's Lemma (WKL) and develop suitable machinery including a version of the Friedberg jump-inversion theorem. The main theorem is obtained when we combine these with known results about WKL. In an appendix we give a proof of the $\Pi^1_1$ conservativity of WKL over $\RCA$ by way of the Superlow Basis Theorem and a new proof of a recent jump-inversion theorem of Towsner.

\section{Introduction}

 The \emph{Cohesiveness Principle} (COH) is a combinatorial statement saying that there exist sets which are infinite, but which cannot easily be split into two infinite subsets. To be precise:

 \begin{description}
  \item [COH:]
  If $\vec R = \langle R_0,R_1,\ldots,\rangle$ is a sequence of sets of natural numbers then there is an infinite set $C$, called an \emph{$R$-cohesive set}, such that for every $k$, either all but finitely many $x \in C$ are in $R_k$, or all but finitely many $x \in C$ are in its complement $\overline{R_k}$.
 \end{description}

This principle arises naturally in the study of Ramsey's theorem for pairs: in a series of papers beginning with Cholak, Jockusch, and Slaman \cite{CJS} and ending (some might say) in Mileti's thesis \cite{Mileti:thesis}, Ramsey's theorem has been shown to decompose into the reverse-mathematical elements $B\Sigma_2$ (an induction axiom), $\COH$, and $D^2_2$ (a pigeonhole principle). Although these results are quite recent, the Cohesiveness Principle itself can be found in earlier work in recursion theory, namely that relating to the \emph{cohesive}, the \emph{r-cohesive}, and the \emph{p-cohesive} sets studied as part of Post's Programme. While the connection between reverse mathematics and recursion theory is well-known, it is used mainly for the construction of counterexamples, or used in a noneffective way (such as carrying out a recursion-theoretic forcing `outside the model'). We are interested in \emph{positive} constructions that preserve effective content wherever possible. As a first example, consider the following theorem which is suggested by known facts about the Turing degrees of $p$-cohesive sets.

\begin{theorem*}
 COH is equivalent over the axiom system $\RCA + \BSt$ to the statement, Every infinite $\Delta^0_2$ binary tree has an infinite $\Delta^0_2$ path. (Here $\Delta^0_2$ means the pointwise limit of a sequence of trees or paths.)
\end{theorem*}

We state an expanded version of this as Theorem \ref{thm:3}, and prove it in Section \ref{sec:4}. Thus, at least in the presence of $\BSt$, the Cohesiveness Principle is a cousin to Weak K\"{o}nig's Lemma ($\WKL$), a principle which makes the same claim but for individual trees and paths rather than for limiting approximations to them. In this article we show several ways the recursion-theoretic outlook can provide positive results such as the above in the setting of second-order arithmetic. Naturally, we must develop a certain amount of machinery to translate our intuition from computability to combinatorics. Most of the effort will be towards our Main Theorem, a family of conservation results that we present below as Theorem \ref{thm:1}. In the Appendix we use similar methods to give alternative proofs of several known results, including conservation for $\WKL$ using the Superlow Basis Theorem (Appendix \ref{app:A}), and a new proof a recent jump-inversion theorem of Towsner (Appendix \ref{app:B}).

An axiom $\phi$ is called \emph{$\Pi^1_1$ conservative over} an axiom system $\Gamma$ if every $\Pi^1_1$ sentence provable from $\Gamma \cup \{\phi\}$ is provable from $\Gamma$ alone. We can similarly define \emph{conservative over} with an axiom system $\Phi$ is place of a single $\phi$, or with another class of sentences in place of $\Pi^1_1$. Alongside proof-theoretic ordinals, consistency strength, and a few other instruments, conservation theorems give a way to measure the closeness or difference in power between one axiom system and the next. Conservation theorems for $\Pi^1_1$ sentences are popular in reverse mathematics, partly because all arithmetical sentences are $\Pi^1_1$, and partly because quite often the proof adapts a straightforward and well-known template (as will be the case in this article). We will prove the following.

\begin{theorem} \label{thm:1}
 For any $n \geq 2$, $\COH$ is $\Pi^1_1$ conservative both over $\RCA + \ISn$ and over $\RCA + \BSn$.
\end{theorem}

The $\ISt$ and $\BSt$ cases have been shown previously by Cholak, Jockusch, and Slaman \cite[Thm 10.2]{CJS} and by Chong, Slaman and Yang \cite[Cor 3.1]{CSY:2012}, respectively. Our methods are different from those in both \cite{CJS} and \cite{CSY:2012}. An advantage of the present proof over than in \cite{CSY:2012} is that it does not require a separate argument for models of $\neg \ISt$. The cases from $n = 3$ upward are new. Although Cholak, Jockusch and Slaman have also shown in \cite[Thm 9.1]{CJS} that $\COH$ is $\Pi^1_1$ conservative over $\RCA$ ($+\ISo$), our method does not apply to this case and so we do not include it in the Theorem.

Our proof, at its core, follows a very standard template: To show that $\phi$ is $\Pi^1_1$ conservative over $\Gamma$, find a way to expand a model of $\Gamma$ to a model of $\Gamma \cup \{\phi\}$ without changing its first-order part. Then no new $\Pi^1_1$ sentences are true in the new model, and the theorem is proved. Our argument is somewhat unusal in that we take an explicitly recursion-theoretic view of the process. That is, we look at a model $\CM$ as consisting of sets which are $\Delta^0_1$ (specifically, relative to themselves) and try to handle sets outside of the model using recursion-theoretic shortcuts such as Post's theorem or the low basis theorem. In Section \ref{sec:2} we develop enough degree theory in the context of $\RCA$ to prove our main theorem, and in the Appendix we develop some of the theory of the $\omega$-r.e.\ degrees.

A second interesting feature of our argument is its (in places) essential use of the weak axiom system $\RCAstar$. This system was introduced by Simpson and Smith in \cite{SS:1986:FOP}, and developed by Enayat, Hatzikiriakou, Simpson, Wong, Yokoyama, and others as a second base system to use alongside or in place of $\RCA$. However, we do not know of other work where $\RCAstar$ was used to prove an equivalence over $\RCA$. \footnote{Between the time this article was written and the time it was placed on the Arxiv, some progress has been made in this direction, for example in work of Ko\l odziejczyk, Kowalik and Yokoyama \cite{KKY:2021:HSI}.}

\subsection{The plan}
We begin with the intuition behind the proof of Theorem \ref{thm:1}. Due to the amount of notation involved, we give many of the definitions only afterwards, in Section \ref{sec:1.2}. We restrict ourselves for the time being to the $\BSt$ case of the theorem, the others being similar. Our starting point is the following characterization of $\COH$ for $\omega$-models, which is implicit in the work of Jockusch and Stephan on $p$-cohesive Turing degrees. (As usual, an \emph{$\omega$-model} is a model $\CM = (M,\CS)$ in which the first-order part $M$ is equal to the set $\omega$ of natural numbers.)

\begin{theorem}[Essentially Jockusch and Stephan \cite{JS:93:ACS}] \label{thm:2}
A nonempty ideal $\CI$ in the Turing degrees is an $\omega$-model of $\COH$ if and only if for each $\ba \in \CI$ there is a $\bb \in \CI$ such that $\bb'$ is PA relative to $\ba'$. 
\end{theorem}
(Here an \emph{ideal} means a collection of Turing degrees closed under the join operation and downward-closed in the usual ordering $\leqT$; $\ba'$ is the Turing jump of $\ba$; and for $\by$ to be PA relative to $\bx$ means $\by$ computes an element of every nonempty $\Pi^0_1(\bx)$ class.) It is simple to derive this theorem from a relativized version of \cite[Thm 2.1]{JS:93:ACS} and the well-known characterization of $\COH$ in terms of relatively $p$-cohesive degrees. Theorem \ref{thm:2} gives a step-by-step recipe to construct $\omega$-models of $\COH$:

\begin{enumerate}
 \item Start with a principal ideal $\CD(\leq \ba)$.
 \item Find a degree $\bp$ that is PA over $\ba'$.
 \item Find a degree $\bb \geq \ba$ such that $\bb' = \bp$. (Such a $\bb$ exists by the Friedberg jump theorem.)
 \item Extend to the principal ideal $\CD(\leq \bb)$, and repeat.
\end{enumerate}

This extends readily with requirements such as cone avoidance to produce, for instance, $\omega$-models of $\COH + \neg \WKL$. We would like to adapt the recipe to work for non-$\omega$-models. Before us are several hurdles, the first of which is a need for a non-$\omega$-model version of Theorem \ref{thm:2}. This need is met by the following, which draws also from work of Friedman, Simpson, and Smith on the equivalence between $\WKL$ and the $\Sigma^0_1$ separation principle.

\begin{theorem}[After Jockusch--Stephan \cite{JS:93:ACS} and Friedman--Simpson--Smith \cite{FSS:1983:CAA}] \label{thm:3}
 Over $\RCA + \BSt$, the following are equivalent:
 \begin{enumerate}
  \item[\rm (i)] $\COH$.
  \item[\rm (ii)] The $\Sigma^0_2$ separation principle: If $A_0,A_1$ are disjoint $\Sigma^0_2$ sets of natural numbers, there is a $\Delta^0_2$ set $D$ such that $A_0 \subseteq D \subseteq \overline{A_1}$. (Here $\overline{A_1}$ denotes the complement of $A_1$.)
  \item[\rm (iii)] Every infinite $\Delta^0_2$ binary tree has an infinite $\Delta^0_2$ path.
 \end{enumerate}
\end{theorem}

 See Section \ref{sec:4} below for a proof. Here $\Sigma^0_n,\Delta^0_n$ are intended to mean the `boldface' versions of the respective classes, taking one arbitrary parameter from the second-order part of the model. Similarly $\Sigma_n,\Delta_n$ are the `lightface' versions taking no second-order parameters, and $\Sigma^A_n,\Delta^A_n$ take the parameter $A$ and no others. Through Theorem \ref{thm:3} we may think of $\COH$ as a `jumped' version of $\WKL$. We can use this to rewrite our earlier recipe to work with models of $\RCA + \BSt$. Because there is a step where we cannot assume $\ISo$, we reason in $\RCAstar$ rather than $\RCA$. Recall that $\RCA + \BSt$ is equivalent to $\RCAstar + \BSt$, and $\RCAstar + \BSo$ is the same as just $\RCAstar$.

\begin{enumerate}
 \item Start with a model $\CM \models \RCAstar + \BSt$ topped by a set $A$, and choose an infinite $\Delta^0_2$ binary tree $T$. (Here \emph{topped} means there is an $A$ in $\CM$ such that every other set in $\CM$ is $\Delta_1^A$.)
 \item \label{recipe:2:ii} Then $\CM[A'] \models \RCAstar + \BSo$. Find a path $P$ through $T$ such that $\CM[P \oplus A'] \models \RCAstar + \BSo$.
 \item \label{recipe:2:iii} Find $B$ such that $P$ is $\Delta_2^B$ and $\CM[B] \models \RCAstar + \BSt$.
 \item Extend to $\CM[B]$, and repeat.
\end{enumerate}

Before we can put this recipe to use, we must overcome three further hurdles: we need to check that $\CM[A']$ does indeed satisfy $\BSo$ in step \ref{recipe:2:ii}, we need a way of producing the $P$ in step \ref{recipe:2:ii}, and we need a way of producing the $B$ in step \ref{recipe:2:iii}. The bulk of this article is devoted to these hurdles. First, in Section \ref{sec:2}, we develop the theory of Turing degrees in the context of $\RCAstar$ up to the point where we can justify the claim in \ref{recipe:2:ii} and obtain, in Section \ref{sec:2.2}, the jump-inversion theorem needed for step \ref{recipe:2:iii}. In Section \ref{sec:3}, we present a number of known results about Weak K\"{o}nig's Lemma that give us the $P$ for step \ref{recipe:2:ii}. Section \ref{sec:3} proves Theorem \ref{thm:3}, linking the construction to $\COH$. Section \ref{sec:5} formally combines the recipe with the ingredients to complete the proof of Theorem \ref{thm:1}. Finally, in the Appendix, we develop some more degree theory over models of $\RCAstar$ and $\RCA$ and give a quick set of proofs of conservation for $\WKL$ and $\COH$.

\subsection{Definitions} \label{sec:1.2}

We use the language $\langle +,\cdot,0,1,<,\in\rangle$ of second-order arithmetic, and much of the usual language of reverse mathematics as laid out in Simpson's book \cite[Ch.~I]{Simpson:2009:Subsystems_of_Second_Order_Arithmetic}. As usual, we are interested in two-sorted models $\CM = (M, \CS)$, where $M$ is the \emph{first-order part}, whose elements are \emph{numbers}, and which interprets $+,\cdot,0,1,$ and $<$; where $\CS$ is the \emph{second-order part}, whose elements are \emph{sets in $\CM$} or \emph{$\Delta^0_1$ sets} and which are identified with elements of the power set $\CP(M)$; and where $\in$ is interpreted as the membership relation between elements of $M$ and elements of $\CS$.  The numbers satisfy the axiom system $P^-$ consisting of the basic rules of Peano arithmetic, including, for example, the distributivity of $\cdot$ over $+$, but not including the axioms of induction. (See Hajek and Pudlak \cite[Ch.~V]{Hajek-Pudlak:1998:Metamathematics_of_first_order_arithmetic} for explicit axiomatizations of $P^-$.) The $\Delta^0_1$ sets obey the axiom of extensionality (which is why we are able to identify them with elements of $\CP(M)$).

We say that a set $A \subseteq M$ is \emph{bounded} if there is a $b \in M$ such that $a \in A$ implies $a < b$. Otherwise, $A$ is \emph{unbounded}. A bounded set $A$ is \emph{finite} if it can be coded canonically as a natural number $c \in M$ (typically by way of its binary expansion). In this article we use \emph{infinite} as a synonym for \emph{unbounded}.\footnote{This is in general a bad definition for \emph{infinite} when considering non-$\omega$-models, as there can be unbounded sets which nonetheless are bounded in size.} For each $n$ and each $A,B \subseteq M$, we say $B$ is a \emph{$\Sigma_n^A$ set} (resp.\ a \emph{$\Delta_n^A$ set}) if $B$ is $\Sigma_n$ definable ($\Delta_n$ definable) with first-order parameters from $M$ and with $A$ as a second-order parameter. We call $B$ a \emph{$\Sigma^0_n$ set} (resp.\ a \emph{$\Delta^0_n$ set}) if $B$ is a $\Sigma^A_n$ set (a $\Delta^A_n$ set) for some $A \in \CS$. All our models $\CM$ will obey the \emph{$\Delta^0_1$ comprehension axiom}:

\begin{description}
  \item [${\Delta^0_1}$ comprehension ($\Deltacomp$):] If $A \subseteq M$ is $\Delta^0_1$ then $A \in \CS$.
\end{description}

\begin{remark}
 (i) Since all our models $\CM$ satisfy $\Delta^0_1$ comprehension, a set $A \subseteq M$ is $\Delta^B_1$ for some $B \in \CS$ if and only if $A$ is in $\CS$. In particular, there is no conflict between the definitions of \emph{$\Delta^0_1$ set} in the first and second paragraphs of this section.

 (ii) Our reason for preferring the term \emph{$\Delta^0_1$ set} rather than simply \emph{set} for elements of $\CS$ is that much of our focus will be on sets outside of the model, which we will usually classify by their level of definability, for instance the $\Delta^0_2$ trees and paths in Theorem \ref{thm:3} above.
\end{remark}

Given a model $\CM = (M, \CS)$ and a set $A \subseteq M$, the \emph{extension of $\CM$ by $A$} is $\CM[A] = (M, \CS^*)$ where $\CS^* \subseteq \CP(M)$ is the transitive closure of $\CS \cup \{A\}$ under the join operation $\oplus$ and $\Delta^0_1$ comprehension.

In addition to the axioms of $P^-$, our models will satisfy certain induction axioms. We phrase these in terms of \emph{regular sets}, which are those $A \subseteq M$ whose every initial segment $A \rest n = A \cap [0,\ldots,n-1]$ is finite (in the sense defined above). (The term \emph{regular} is from $\alpha$-recursion theory. Regular sets are also known by several other names, including \emph{amenable} and \emph{piecewise coded}.) For our purposes, the most important induction axioms are:

\begin{description}
 \item [Induction for $\Sigma^0_n$ sets ($\ISn$):] If $A$ is $\Sigma^0_n$ then $A$ is regular.
 \item [Induction for $\Delta^0_n$ sets ($\IDn$):] If $A$ is $\Delta^0_n$ then $A$ is regular.
\end{description}
The latter is often phrased in the following bulkier, but equally useful, form.
\begin{description}
 \item [Bounding for $\Sigma^0_n$ sets ($\BSn$):] If $\phi$ is a $\Sigma^0_n$ formula with parameters from $\CM$ and $x_0$ is an element of $M$, then $(\forall x< x_0)(\exists y)\phi(x,y)$ implies $(\exists y_0)(\forall x < x_0)(\exists y < y_0) \phi(x,y)$.
\end{description}

It is a theorem of Slaman \cite{slaman:2004} that the two are equivalent in the context of purely first-order arithmetic:

\begin{proposition}[Slaman]
 $P^- + \mathrm{I}\Sigma_0 + \exp \proves \mathrm{I}\Delta^0_n \leftrightarrow \mathrm{B}\Sigma_n$ for every $n \geq 1$, where $\exp$ is an axiom stating that the exponential function $x \mapsto 2^x$ is total.
\end{proposition}

The same holds over $\RCA$ when second-order parameters are allowed. The schemes of $\Sigma^0_n$ induction and $\Sigma^0_n$ bounding were introduced by Paris and Kirby \cite{PK:78:CSI}, who proved that they formed a strict hierarchy of implications: For every $n \geq 0$, $\ISnp$ implies $\BSnp$, which in turn implies $\ISn$. Two possible base axiom systems for reverse mathematics are:

\begin{description}
 \item [Recursive comprehension axiom scheme ($\RCA$):] $P^- + \Deltacomp + \ISo$;
 \item [Recursive comprehension-star ($\RCAstar$):] $P^- + \Deltacomp + \ISz + \exp$.
\end{description}
 
The second system, $\RCAstar$, is strong enough to prove $\BSo$ but not $\ISo$. Hence $\RCA$ is strictly stronger than $\RCAstar$. Now we give two combinatorial principles, $\WKL$ and $\COH$. Although we have already defined $\COH$ once at the beginning of the article, we restate it here in the language of this section so there is no misunderstanding. We use $2^{<M}$ to denote the set of all finite binary strings in a given model $(M,\CS)$, coded as elements of $M$ (for example by prepending 1 to the string and then viewing it as a number written in base two). A \emph{binary tree} is a subset of $2^{<M}$ that is closed under initial segment. A \emph{path} is a binary tree that is totally ordered by end-extension; but we often identify an infinite path with the unique infinite string extending all of its members. A sequence $\langle A_0,A_1,\ldots\rangle$ of sets $A_i \subseteq M$ is \emph{uniformly $\Delta^0_1$} if the set $\{\langle x,i\rangle : x \in A_i\}$ is $\Delta^0_1$. (Here $\langle ., . \rangle$ is the usual quadratic pairing function $M \times M \rightarrow M$.)

\begin{description}
 \item [$\WKL$:] If $T$ is an infinite $\Delta^0_1$ binary tree, then $T$ contains an infinite $\Delta^0_1$ path $P$.
 \item [$\COH$:] If $\vec R = \langle R_0,R_1,\ldots\rangle$ is a uniformly $\Delta^0_1$ sequence of sets then there is a $\Delta^0_1$ infinite $\vec R$-cohesive set $C$.
\end{description}

 Kleene's T-predicate construction gives us a `universal' $\Sigma_1$ formula $\varphi^Z(e,x)$ with free first-order variables $e,x$ and free second-order variable $Z$ such that, whenever $(M,\CS)$ is a model of $\RCAstar$ and $\psi^Z(x)$ is a formula with first-order parameters from $M$ and all regular $A \subseteq M$,
  $$\CM \models (\exists e)(\forall x)[ \varphi^A(e,x) \leftrightarrow \psi^A(x)].$$
 The construction also yields a `universal' $\Sigma_1$ partial function $\Phi^Z(e,x)$ with variables $e,x,Z$ such that for every $\Sigma_1$ partial function $\Psi^Z(x)$ with parameters from $M$ and every regular $A \subseteq M$,
  $$\CM \models (\exists e)(\forall x)[\Phi^A(e,x) \halts \leftrightarrow \Psi^A(x)\halts,~\mathrm{and~their~values~are~equal}\}.$$
 Here we use $f(x)\halts$ to mean $f(x)$ has some value $y$, and $f(x)\nohalts$ to mean it does not. Write $\Phi_e^A$ to mean the partial $\Sigma_1^A$ function given by $\Phi_e^A(x) = \Phi^A(e,x)$. Each $\Phi^A_e$ has a $\Sigma^A_0$ stage-by-stage approximation $(\Phi^A_{e,s})_{e \in M}$ as usual. The \emph{jump} of $A$ is the $\Sigma_1^A$ set $A' = \{e \in M : \Phi_e^A(e) \halts\}$. We write $A^{(n)}$ to mean $A$ with the jump operation iterated $n$ times, but when convenient, we may still write $A''$ and so on.

\section{Basic recursion theory over a model of $\RCAstar$}\label{sec:2}

\emph{Reverse recursion theory} is one name for the practice of asking what subsystems of first-order Peano arithmetic are necessary and sufficient to prove theorems of classical recursion theory. Although typically the focus is on the inductive constructions associated with the recursively enumerable degrees (finite injury, permitting, $0''$ arguments and so on---see Chong, Li, and Yang \cite{CLY:2014:survey} for a recent survey), there has also been work on the more general theory of degrees. Some examples are Hajek's work on the Limit Lemma and the Low Basis Theorem (see \cite[Ch.~1.3]{Hajek-Pudlak:1998:Metamathematics_of_first_order_arithmetic}), and Towsner's work on jump inversion (see Appendix B). In this section we present a number of recursion-theoretic facts, formalized in the language of arithmetic. We view as recursive the sets which are $\Delta^0_1$ in a given model $\CM$, as r.e.\ those which are $\Sigma^0_1$ (i.e.\ $\Sigma_1$ with a second-order parameter from $\CM$), and so on.

Although there are many theorems that can be formalized in this way, we limit ourselves in this article to what is needed to prove Theorem \ref{thm:1}. The most interesting is the jump-inversion theorem in Section \ref{sec:2.2}. Later, in the Appendix, we formalize a few more, mostly having to do with $\omega$-r.e.\ degrees, to prove certain conservation results for $\WKL$. It is possible that a systematic development of degree theory in nonstandard models, perhaps drawing from $\alpha$-recursion theory, would clarify other problems in reverse mathematics.

\subsection{Models of $\RCAstar$ and their extensions}

Most of our constructions pass between models $\CM \models \RCAstar$ and extensions $\CM[G] \models \RCAstar$ while preserving properties such as induction (in varying amounts) or adding new sets such as paths through a tree. Gathered here are a number of technical lemmas saying, essentially, that $\CM[G]$ behaves as one would hope or expect.

\begin{lemma}
 $\RCAstar \proves \BSo$
\end{lemma}
\begin{proof}
 Suppose $\CM = (M,\CS)$ is a model of $\RCAstar$ and $\CM \models (\forall x<a)(\exists y)\psi(x,y)$, with $a \in M$ and $\psi$ a $\Sigma^0_0$ formula. Let
  $$B = \{F \subseteq \{0,\ldots,a-1\} : F~\mathrm{is~finite,~and}~(\exists y_0)(\forall x<a)[x \in F \rightarrow (\exists y < y_0) \psi(x,y)]\}.$$
 Then $B$ is bounded (since $B$ is contained in the power set of $\{0,\ldots,a-1\}$, which is bounded by $\exp$) and $\Delta^0_1$, so it is finite by $\Deltacomp$. It follows by $\ISz$ that $B$ has an element of size $a$, namely $\{0,\ldots,a-1\}$ itself.
\end{proof}

Together with Proposition \ref{prop:3.1} below, this establishes $P^- + exp + B\Sigma_1$ as the first-order part of $\RCAstar$.

 \begin{lemma}\label{lem:2.2}
  Suppose $M \models P^- + \exp$ is a model and $A \subseteq M$ is such that $(M,\{A\}) \models \IDo$. Then:
  \begin{enumerate}
   \item[\rm (i)] If $B$ is $\Delta^A_1$ and $C$ is $\Sigma^B_1$, then $C$ is $\Sigma^A_1$.
   \item[\rm (ii)] If $B$ is $\Delta^A_1$ and $C$ is $\Delta^B_1$, then $C$ is $\Delta^A_1$.
  \end{enumerate}
 \end{lemma}
 \begin{proof}
  (i) Fix $\Sigma_1$ formulae $\phi_0,\phi_1,\psi$ satisfying:

  $\begin{array}{ccc}x \in B &\iff& (\exists s)\phi^A_1(x,s),\\
    x \not \in B &\iff& (\exists s)\phi^A_0(x,s),\\
    y \in C &\iff& (\exists t)\psi^B(y,t).
   \end{array}$

  Assume for simplicity that $\phi^A_1(x,s)$ implies $\phi^A_1(x,t)$ for all $t > s$, and similarly for $\phi^A_0$ and $\psi^B$. Combining the first two formulae, applying $\BSo$ and invoking $\exp$ and the regularity of $A$, we get:
  $$(\forall x_0)(\exists \sigma \in 2^{x_0})(\exists s)(\forall x < x_0)[\sigma(x)=0 \leftrightarrow \phi^A_0(x,s)~\mathrm{and}~\sigma(x)=1\leftrightarrow\phi^A_1(x,s)].$$
  In particular, any initial segment of $B$ is $\Sigma_1$ in a finite initial segment of $A$, and hence $B$ is regular. Now since the third formula, if true, uses only a finite initial segment of $B$:
  $\begin{array}{ccl}
   y \in C &\iff& (\exists t)(\exists \sigma \subseteq B~\mathrm{finite})\psi^\sigma(y,t) \\
           &\iff& (\exists t)(\exists \sigma~\mathrm{finite})[\psi^\sigma(y,t)~\mathrm{and}~(\forall x < |\sigma|)(\forall i < 2)(\sigma(x)=i \leftrightarrow \phi^A_i(x,t))]. 
   \end{array}$

   (ii) Apply part (i) both to $C$ and to its complement $\bar C$.
 \end{proof}
  Because our definition of $\Sigma^0_1$ admits only one second-order parameter, the Lemma above relativizes easily to $\CM \cup \{A\}$ with $\CM \models \RCAstar$ in place of $M$. A thorny detail arises, however, when we try to pass to $\CM[A]$, namely: We do not know whether there might be $B \in \CM$ such that $\BSo$ fails relative to the join $A \oplus B$. Fortunately, it is enough for our purposes to consider only $A$ which are \emph{above} $\CM$, in the sense that every $B \in \CM$ is $\Delta^A_1$.

 \begin{proposition} \label{prop:2.3}
  \begin{enumerate}
   \item[\rm (i)]
  If $\CM \models \RCA$ and $A$ is above $\CM$ and $\CM \cup \{A\} \models \ISo$, then $\CM[A] \models \RCA$.
   \item[\rm (ii)]If $\CM \models \RCAstar$ and $A$ is above $\CM$ and $\CM \cup \{A\} \models \IDo$, then $\CM[A] \models \RCAstar$.
  \end{enumerate}
 \end{proposition}
 \begin{proof}
  (i) The axioms $P^-$ and $\exp$ are purely first-order and so are satisfied. As for $\ISo$, by Lemma \ref{lem:2.2} (i) every $\Sigma^0_1$ set in $\CM[A]$ is $\Sigma^A_1$ and hence regular.

  (ii) Similar to the above; by Lemma \ref{lem:2.2}(ii) every $\Delta^0_1$ set of $\CM[A]$ is $\Delta^A_1$ and hence regular.
 \end{proof}
 We are especially interested in expanding models by adding jumps of certain sets. For this we use the following formalization of Post's Theorem for non-$\omega$-models.
\begin{proposition}[Post's Theorem] \label{prop:2.4}
 Fix $M \models P^- + \exp$ and $A \subseteq M$.
 If $(M, \{A\}) \models \ISo$ then for all $n \geq 1$ the $\Sigma_{n+1}^A$ sets are exactly the $\Sigma_n^{A'}$ sets.
\end{proposition}
\begin{proof}
 First we show that every $\Sigma_{n+1}^A$ is $\Sigma_n^{A'}$. Suppose first that $n$ is even, and let $B$ be any $\Sigma_{n+1}$ set, say with $x \in B$ iff $(\exists y_1)(\forall y_2) \cdots(\exists y_{n+1}) \psi^A(x,y_1,\ldots,y_{n+1})$, where $\psi$ is $\Delta^A_0$. Since $\ISo$ holds, for any $a,b_1,\ldots,b_{n} \in M$ we have $(\exists y_{n+1}) \psi^A(a,b_1,\ldots,b_{n},y_{n+1})$ iff $(\exists y_{n+1})(\exists \sigma \subseteq A) \psi^\sigma(a,b_1,\ldots,b_{n},y_{n+1})$. Thus we can define a computable $h$ on $n+1$ variables such that this holds iff $h(a,b_1,\ldots,b_{n}) \in A'$. Then
   $$x \in B \iff (\exists y_1)(\forall y_2) \cdots (\forall y_{n}) h(x,y_1,\ldots,y_{n}) \in A',$$
  which is $\Sigma_n^{A'}$, as desired. The case for odd $n$ is similar, producing an $h$ such that
   $$x \in B \iff (\exists y_1)(\forall y_2) \cdots (\exists y_{n}) h(x,y_1,\ldots,y_{n}) \not \in A'.$$
 For the converse, suppose first that $n$ is odd, and let $B$ be $\Sigma^{A'}_n$, say with $x \in B$ iff $(\exists y_1)(\forall y_2)\cdots(\exists y_n)\psi^{A'}(x,y_1,\ldots,y_n)$. By $\ISo$ we know that $A'$ is regular, hence
   $$x \in B \iff (\exists y_1) \cdots (\forall y_{n-1})(\exists \sigma)(\exists s)\left[ \psi^\sigma(x,y_1,\ldots,y_n)\mathrm{~and}~(\forall t > s)(\forall e < |\sigma|)\Phi_{e,t}^A(e) \halts \leftrightarrow \sigma(e)=1\right],$$
 which is $\Sigma_{n+1}^A$. [There is a subtlety here worth mentioning: When $\Phi_e^{A'}(e) \halts$, why can we guarantee that $\Phi_e^{\sigma}(e) \halts$ for some finite initial segment $\sigma$ of $A'$? Because we can find, by recursion on the bounded-quantifier complexity of the $\Delta_0$ matrix of $\Phi_e$, computable upper bounds for the possible values of the terms in it (including e.g.~use of the exponential operations implicit in the use of oracles and strings)]. The case for even $n$ is similar.
 \end{proof}
This proposition can be iterated to get results about $\Sigma_n^{A''}$ sets in a model of $\ISt$, and so on. From Propositions \ref{prop:2.3} and \ref{prop:2.4} we obtain quick corollary.
 \begin{corollary}
  Fix $n \geq 1$, $M \models P^- + \exp$ and $A \subseteq M$. Then
  \begin{enumerate}
   \item If $M \cup \{A\} \models \IDnp$ then $M[A'] \models \IDn$.
   \item If $M \cup \{A\} \models \ISnp$ then $M[A'] \models \ISn$.
  \end{enumerate}
 \end{corollary}
 This corollary can also be iterated to obtain models $\CM[A'']$ and so on, with the expected levels of induction.

\subsection{The Friedberg jump theorem} \label{sec:2.2}

We will formalize the following well-known theorem of recursion theory.
\begin{theorem*}[Friedberg]
 If $\bc$ is a Turing degree, there is a Turing degree $\bb$ such that $\bb' = \bzero' \oplus \bc$.
\end{theorem*}
This is sometimes called a \emph{jump inversion theorem}, because it provides a one-sided inverse to the jump operator $\bb \mapsto \bb'$, or a \emph{completeness criterion}, because it characterizes the Turing-complete degrees, i.e.\ those above $\bzero'$.  The proof we give below has the advantage that it works for uncountable models. On the other hand, it also requires extra hypotheses about the model and about the set $C$ (correponding to $\bc$). Towsner in \cite{towsner:2015:OMC} has recently proved a version for $\ISn$ that swaps these extra hypotheses out in favour of a requirement that the model $\CM$ be countable (see Appendix \ref{app:B} for details and a proof). We begin with a weak version, and then give a more complete version as Proposition \ref{prop:FriedbergFull}
\begin{lemma}\label{lem:2.6}
 Let $M \models P^- + I\Sigma_1$, and suppose $A'\oplus C$ is regular. Then there is a $B$ such that $A' \oplus C$ and $(A \oplus B)'$ are $\Delta_1$ relative to each other and $(A \oplus B)'$ is regular; in particular, $M[A \oplus B] \models \ISo$. Furthermore, if $M[A' \oplus C] \models \BSo$ then $M[A \oplus B] \models \BSt$.
\end{lemma}
Our proof is less straightforward than Friedberg's original, in order to accommodate the relatively weak axiom systems involved.
\begin{proof}
 We build $B$ and $(A \oplus B)'$ by initial segments $\sigma_0 \subseteq \sigma_1 \subseteq \cdots \subseteq B$ and $\tau_0 \subseteq \tau_1 \subseteq \cdots \subseteq (A \oplus B)'$. We proceed by stages $s\in M$. At each stage $s>0$ we are performing either \emph{Strategy $B$} to build $B$, or in \emph{Strategy $(A \oplus B)'$} to build the jump.

 Stage $s=0$. Let $\sigma_0$ and $\tau_0$ be the empty string. Change to Strategy $B$ and go to the next stage $s+1$.

 Strategy $B$. Let $\tau_s = \tau_{s-1}$. Check whether there is a string $\sigma$ of length $s$ extending $\sigma_{s-1}$ satisfying $(\forall e < |\tau_s|)[\tau_s(e) = 0 \rightarrow \Phi_{e,s}^{A \oplus \sigma}(e) \halts]$.  If so, let $\sigma_s$ be the concatenation of the lexicographically-least such $\sigma$ and the initial segment $C \rest s$ (viewing $C$ as a binary string), and change to Strategy $(A \oplus B)'$. Otherwise let $\sigma_s = \sigma_{s-1}$. Go to stage $s+1$.

 Strategy $(A \oplus B)'$. Let $\sigma_s = \sigma_{s-1}$. Let $\tau_s$ be the lexicographically-least $\tau$ of length $s$ extending $\tau_{s-1}$ and satisfying $(\exists \sigma \supseteq \sigma_{s})(\exists t)(\forall e < s)[\tau(e)=0 \rightarrow \Phi^{A \oplus \sigma}_{e,t}(e) \halts]$. Change back to Strategy $B$ and go to the next stage $s+1$. This completes the construction.

 Now we make a few observations. The first is that this construction is effective in $A' \oplus C$, and at each stage $s$ looks only computably-far forwards in the oracle. This is because $\exp$ is all that's needed to form the set of strings of a given length, and there is a computable function $h$ satisfying $\Phi_e^{A \oplus \sigma}(e) \halts$ iff $h(x,\sigma) \in A'$.
  The second is that, as a result, $\sigma_s,\tau_s$ are defined for all $s \in M$, as opposed to just on a proper cut. The third is that for every stage $s$, there is a stage $t > s$ at which Strategy $B$ finds a proper extension $\sigma_t$. Then $\sigma_t$ has length $\geq t >s$ (from $C \rest t$); hence $B$ is well-defined as a subset of $M$, as opposed to just on a proper cut. The fourth is that we also have $|\tau_t| = t$ cofinally often, so that the set being constructed is well-defined and \emph{really is} the jump $(A \oplus B)'$; and in particular $(A \oplus B)'$ is $\Delta^{A' \oplus C}_1$. The fifth is that $C$ is $\Delta_1^{A' \oplus B}$, by using $A'$ and $B$ together to find the stages $s$ at which $\sigma_s \subseteq B$ increases in size, and then looking at the elements representing $C \rest s$. It follows that $(A' \oplus C)$ is $\Delta_1^{A' \oplus B}$ and hence $\Delta_1^{(A \oplus B)'}$.

  This completes the proof of the main part of the theorem. The `in particular' part follows because if $D$ is a $\Sigma_1^{A \oplus B}$ set then $D$ is $m$-reducible to $(A \oplus B)'$, a regular set, and hence is itself regular. The `furthermore' part is by Lemma \ref{lem:2.2}(ii).
\end{proof}
 A more general result follows when after relativizing to the appropriate $0^{(n)}$:
\begin{proposition}\label{prop:FriedbergFull}
 Let $M[A^{(n)}\oplus C] \models P^- +\exp + \ISz$ (resp.~$P^- + \exp + \BSo$) with $n \geq 1$. Then there is a $B$ such that $A^{(n)} \oplus C$ and $(A \oplus B)^{(n)}$ are $\Delta_{n}$ relative to each other, and $M[A \oplus B] \models \ISn$ (resp.~$\BSnp$).
\end{proposition}
 Here $M$ is a first-order model, and $M[A^{(n+1)} \oplus C]$ is the second-order model topped by $A^{(n+1)} \oplus C$. We would be interested to know whether a theorem of this form holds for models that are not topped by a single set $A$.
\section{Conservation for Weak K\"{o}nig's Lemma} \label{sec:3}

\begin{proposition}[Simpson--Smith]\label{prop:3.1}
 $\RCAstar + \WKL$ is $\Pi^1_1$-conservative over $\RCAstar$.
\end{proposition}

We present a proof of this proposition because it demonstrates the method of iterated extensions we use later in Section \ref{sec:5}, and because at any rate the following lemma is itself needed later on.
\begin{lemma}[Simpson--Smith]\label{lem:3.2}
 If $\CM \models \RCAstar$ is countable and $T$ is an infinite $\Delta^0_1$ binary tree then $T$ contains an infinite path $P$ such that $\CM[P] \models \RCAstar$.
\end{lemma}
\begin{proof}
 Working outside of the model, enumerate the $\Sigma_1$ formulas $\psi^Y(x)$ with first-order variable $x$, second-order variable $Y$, and first- and second-order parameters from $\CM$, as $(\phi_n)_{n < \omega}$.  (As usual, $\omega$ represents the true natural numbers. In particular, $(\phi_n)_{n<\omega}$ is not required to be an effective enumeration.) For each $n < \omega$ and $x \in M$, let
  $$U_{n,x} = \{\sigma \in 2^{<M} : \neg \phi_n^\sigma(x)\}.$$
 Now define $T_0 = T$, and $T_{n+1} =$ either $T_n \cap U_{n,x}$ if for some $x$ this is infinite, or $T_{n+1} = T_n$ otherwise. Then each $T_n$ is an infinite $\Delta^0_1$ tree in $\CM$. It is easy to see that their intersection $\bigcap_n T_n$ is a single path; call it $P$. Certainly $P$ is regular; we claim that $\CM \cup \{P\}$ satisfies $\BSo$. Fix any $a \in M$ and any $\Sigma_1$ formula $\phi_n^Y(x)$ with first- and second-order parameters from $\CM$ such that $(\forall x < a)(\exists y)\phi_n^P(x,y)$ holds. Then by construction $U_{n,x} \cap T_n$ is finite for each $x < a$. By $\BSo$ in $\CM$, there is an upper bound on their heights. This gives an upper bound on the witnesses $y$, proving the claim. We conclude by Proposition \ref{prop:2.3} that $\CM[P] \models \BSo$.
\end{proof}

\begin{proof}[Proof of Proposition \ref{prop:3.1}.]
 Suppose for a contradiction that there is an arithmetical formula $\phi(X)$ of a single second-order variable such that $\RCAstar + \WKL \proves (\forall X) \phi(X)$ but $\RCAstar \not \proves (\forall X) \phi(X)$. Let $\CM$ be a countable model of $\RCAstar + (\exists X) \neg \phi(X)$, say with witness $\neg \phi(A)$. By iterating Lemma \ref{lem:3.2}, expand $\CM$ to a model $\CM^* \models \RCAstar + \WKL$ with the same order part, and with a second-order part containing that of $\CM$. Since $\neg \phi(A)$ still holds in $\CM^*$, we have $\CM \models \RCAstar + \WKL + \neg (\forall X) \phi(X)$, contradicting the assumption that $\RCAstar + \WKL \proves (\forall X) \phi(X)$.
\end{proof}

We finish this section by stating a more general version of Lemma \ref{lem:3.2} that combines results from several sources. The earliest is the $\ISo$ case, due to Harrington \cite{harrington:77}; the second is the $\BSo$ case, which is the Simpson--Smith \cite{SS:1986:FOP} result above; and most recent are the remaining cases $\ISn$, $\BSn$ for all $n \geq 2$, first published by Hajek \cite{hajek:93:IAF}. We give alternative and, perhaps, simpler proofs of Harrington's and Hajek's cases in the Appendix.
\begin{lemma}[Harrington; Simpson--Smith; Hajek]
 Fix $n \geq 1$. Suppose $\CM \models \RCAstar$, and $T$ is an infinite binary $\Delta^0_1$ tree.
  \begin{enumerate}
   \item If $\CM \models \BSn$, there is an infinite path $P$ such that $\CM[P] \models \BSn$.
   \item If $\CM \models \ISn$, there is an infinite path $P$ such that $\CM[P] \models \ISn$.
  \end{enumerate}
\end{lemma}
\begin{proof}
 In the Appendix.
\end{proof}

Each case of the Lemma yields a conservation result analogous to Proposition \ref{prop:3.1}; the method was introduced by Harrington. Some of the conservation theorems were later revisited by Avigad \cite{avigad:96:FFA} using proof-theoretic techniques, and by Wong \cite{wong:2017:IWK} using techniques from the model theory of Peano arithmetic.

\section{Characterizing $\COH$} \label{sec:4}

This section contains three lemmas which together prove Theorem \ref{thm:3}, characterizing $\COH$ over $\BSt$. Recall once again that $\RCAstar + \BSt$ is equivalent to $\RCA + \BSt$, and $\RCAstar + \ISo$ to $\RCA$.

\begin{lemma}\label{lem:TreesCOH}
 $\RCAstar + \BSt \proves$ `Every infinite $\Delta^0_2$ binary tree has an infinite $\Delta^0_2$ path' $\rightarrow \COH$.
\end{lemma}
\begin{proof}[Proof (after Jockusch--Stephan).]
 Working in a model $\CM \models \RCAstar + \BSt$ satisfying the statement about trees, begin with $\vec R = \langle R_0,R_1,\ldots\rangle$ and define a tree $T$ by:
  $$\sigma \in T \iff (\exists^{> |\sigma|} x)(\forall k < |\sigma|)[R_k(x) \leftrightarrow \sigma(k) =1].$$
 This $T$ is $\Sigma^0_1$, and infinite (since for every $\ell$ there is a $\sigma \in T$ such that $|\sigma|=\ell$, by $\BSo$), thus $T$ has an infinite $\Delta^0_2$ path $P$. Now let $f : M \times M \rightarrow \{0,1\}$ be a $\Delta^0_1$ function such that $P(k) = \lim_s f(k,s)$ for all $k$, and construct a set $C = \{c_1, c_2, \ldots\}$ element-by-element:

 Stage $\ell$. Find the least $s$ such that $(\forall k<\ell)[c_k < s]$, while satisfying:
  $$(\forall k < \ell)[R_k(s) \leftrightarrow f(k,s) = 1],$$
 and let $c_\ell = s$. Such an $s$ must exist by $\BSt$, and since searching for the least such $s$ is a computable process, $\ISo$ implies that $c_\ell$ is defined for all $\ell \in m$ (as opposed to just a proper cut). Now $C$ is infinite, $\Delta^0_1$, and $\vec R$-cohesive, as desired.
\end{proof}

\begin{lemma}\label{lem:COHSep}
 $\RCAstar + \ISo \proves \COH \rightarrow$ The $\Sigma^0_2$ separation principle.
\end{lemma}
\begin{proof}[Proof (after Jockusch--Stephan).]
Working in a model $\CM \models \RCAstar + \ISo + \COH$, fix any pair $A_0,A_1$ of disjoint $\Sigma^0_2$ sets. Let $f : M \times M \rightarrow \{0,1\}$ be a $\Delta^0_1$ function such that for each $i \in \{0,1\}$ and each fixed $x$, the set $\{s : f(x,s) = i\}$ is infinite iff $x \not \in A_i$.
 Now define a sequence $\vec R = \langle R_0,R_1,\ldots\rangle$ of unary relations by:
  $$R_x(s) \iff f(x,s) = 1.$$
 Apply $\COH$ to get an $\vec R$-cohesive set $C$. Now define a set $D$ by:
  $$x \in D \iff (\forall^\infty c \in C)f(x,c) = 1,$$
 or equivalently (because $C$ is $\vec R$-cohesive):
  $$x \not \in D \iff (\forall^\infty c \in C)f(x,c) = 0.$$
 Now $D$ is a $\Delta^0_2$ separating set for $A_0,A_1$, and the lemma is proved.
\end{proof}
\begin{lemma}\label{lem:SepTrees}
 $\RCAstar + \BSt \proves$ The $\Sigma^0_2$ separation principle $\rightarrow$ `Every infinite $\Delta^0_2$ binary tree has an infinite $\Delta^0_2$ path'.
\end{lemma}
\begin{proof}[Proof (after Friedman--Simpson--Smith).]
 Working in a model $\CM \models \RCAstar + \BSt$ that satisfies $\Sigma^0_2$ separation, begin with an infinite $\Delta^0_2$ binary tree $T$, say with approximation $(T_s)_{s \in M}$ where each $T_s$ is a tree, and define two $\Sigma^0_2$ sets $A_0,A_1$ of binary strings by:
  $$\sigma \in A_0 \iff (\exists s)[\mathrm{no~new}~\tau \supseteq \sigma 1~\mathrm{enter}~T_t~\mathrm{at~any~stage}~t>s,~\mathrm{but~a~new}~\tau \supseteq\sigma 0 ~\mathrm{does~enter}],$$
  $$\sigma \in A_1 \iff \mathrm{same,~but~with~}\sigma1~\mathrm{and}~\sigma0~\mathrm{switched.}$$
 These $A_0,A_1$ are $\Sigma^0_2$ and disjoint by definition. Let $D$ be a $\Delta^0_2$ set of binary strings such that $A_1 \subseteq D \subseteq \overline{A_0}$, and define a second set $P$ by:
  $$\sigma \in P \iff \sigma \mbox{~is~lex-least~of~length}~|\sigma|\mbox{~s.t.}~(\forall k < |\sigma|)[\sigma \rest k \in D \leftrightarrow \sigma(k)=1].$$
 This $P$ is $\Delta^0_2$ because $D$ is $\Delta^0_2$ and $\BSt$ holds, contains a string of every length by $\IDt$, and is totally-ordered also by $\BSt$. Thus $P$ describes an infinite $\Delta^0_2$ path through $T$.
\end{proof}

\section{Completing the proof} \label{sec:5}
 \begin{proof}[Proof of Theorem 1.]
  Suppose for a contradiction that $\RCAstar + \BSt + \COH \proves (\forall X) \phi(X)$ but $\RCAstar + \BSt \not \proves (\forall X)\phi(X)$ for some arithmetical formula $\phi$ with free second-order variable $X$. Choose a countable model $\CM = (M,\CS)$ of $\RCAstar + \BSt + (\exists X)\neg \phi(X)$. Let $A \in \CM$ be such that $\CM \models \neg \phi(A)$. By throwing some sets out of $\CS$, we may assume $\CM$ is topped by $A$. Let $T$ be an infinite binary tree which is $\Delta^A_2$. Then $T$ is $\Delta^{A'}_1$; and since $\CM[A'] \models \BSo$, by Lemma \ref{lem:3.2} there is an infinite path $P$ such that $\CM[A'][P] \models \BSo$. Then Lemma \ref{lem:2.6} provides a set $B$ such that $\CM[B] \models \BSt$ and according to which $P$ is $\Delta^0_2$. Treat this $\CM[B]$ as our new topped model, and repeat this procedre for the next $T$.

  Continue for $\omega$-many stages, while dovetailing to ensure that every $T$ is eventually dealt with. In the limit we are left with a model $\CN \models \RCAstar + \BSt$. Since in $\CN$ every infinite $\Delta^0_2$ binary tree has an infinite $\Delta^0_2$ path, we know by Theorem \ref{thm:3} that $\CN \models \COH$. And since $A$ is still an element of $\CN$, we still have $\CN \models (\exists X)\neg \phi(X)$, contradicting our starting assumption that $\RCAstar + \BSt + \COH \proves (\forall X)\phi(X)$.
\end{proof}

\section{Further questions}
We would like to know, first of all, whether the Cohesiveness Principle is also conservative over other systems. From the results in this and other papers, it is known to be conservative over $\RCAstar + \ISn$ and $\RCAstar + \BSn$ for every $n \geq 2$.

\begin{question}
 Is the Cohesivenes Principle $\Pi^1_1$ conservative over $\RCAstar$?
\end{question}
Our other questions are more programmatic. We have seen in Theorem \ref{thm:3} that, at least over $\BSt$, $\COH$ can be rephrased as a version of the `Big Five' principle $\WKL$. An earlier example along the same lines is in Chong, Lempp, and Yang \cite{CLY:2010:OTR}, where it is shown that the \emph{stable Ramsey theorem for pairs,} a strictly combinatorial statement, can be rephrased as a jump-inverting pigeonhole principle:
\begin{description}
 \item ($D^2_2$) If $f$ is a $\Delta^0_2$ total function from the natural numbers to $\{0,1\}$ then $f$ is constant on some infinite $\Delta^0_1$ set.
\end{description}
We would like to know just how common this phenomenon is.
\begin{question}
 What statements studied in reverse mathematics can be phrased more simply in terms on $\Delta^0_n$-definable sets? Does this bring any order to the zoo\footnote{\emph{The Reverse Mathematics Zoo} is a database and software package hosted by the University of Connecticut Department of Mathematics at {\tt http://rmzoo.math.uconn.edu/~.}} of reverse-mathematical principles?
\end{question}
Moving in the other direction, we can try to extend the analogy between $\WKL$, for $\Delta^0_1$ sets, and $\COH$, for $\Delta^0_2$ sets, to some other complexity. If we define an $n$-th analogue fo $\WKL$ by
\begin{description}
 \item ($n$-$\WKL$) Every infinite $\Delta^0_n$ binary tree has an infinite $\Delta^0_n$ path,
\end{description}
it is straightforward to generate theorems, for example, that for each set $A \subseteq \omega$ there is an $\omega$-model of $\bigwedge_{n \in A}n$-$\WKL \wedge \bigwedge_{n \not \in A} \neg n$-$\WKL$, or that $n$-$\WKL$ is $\Pi^1_1$-conservative over the appropriate axiom systems. But it is unclear whether these principles are of mathematical interest otherwise.
\begin{question}
 Is there any use for the principle, `Every infinite $\Delta^0_3$ binary tree has an infinite $\Delta^0_3$ path'?
\end{question}
Another sort of question concerns the choice of base system for reverse mathematics. It is at times necessary to work over $\RCAstar$ instead of $\RCA$. Our question is therefore not whether dropping to $\RCAstar$ can be useful---clearly it can---but whether that is dropping far enough.
\begin{question}
 Can one prove theorems of `ordinary' reverse mathematics (the kind over $\RCA$) by using systems weaker than $\RCAstar$? In particular, is there a good reason to weaken the axiom of $\Delta^0_1$ comprehension?
\end{question}

\appendix
\section{More conservation for Weak K\"onig's Lemma}\label{app:A}
In Section \ref{sec:3} we presented a version of Simpson and Smith's \cite{SS:1986:FOP} proof that $\WKL$ is $\Pi^1_1$-conservative over $\RCAstar$ and mentioned some proofs of Harrington, Hajek, Avigad, and Wong for similar conservation theorems over $\RCA$ and $\RCA + \BSnp$ for all $n \geq 1$.
Hajek's proof for $\ISo$ uses a notion that he calls \emph{very low sets} in \cite{hajek:93:IAF}, or \emph{low $\Sigma_0^*(\Sigma_1)$ sets} in \cite[A.2.d,A.3.b]{Hajek-Pudlak:1998:Metamathematics_of_first_order_arithmetic}. In this section we give a proof that is largely the same---indeed, largely the same as the original proof of Jockusch and Soare \cite{Jockusch-Soare:1972:CAD}---but that uses more familiar language of recursion theory, such as \emph{superlow sets}.

An imporant difference between the proof here and that of the $\RCAstar$ result in Section \ref{sec:3} is that whereas for $\RCAstar$ we used an \emph{external forcing} by taking a bijection between the definable formulas of the model $\CM$ and the true natural numbers $\omega$, and then performing an inductive construction in $\omega$, here we use an \emph{internal forcing} that refers only to $\CM$ (as we did in Lemma \ref{lem:2.6}). The internal approach has a number of advantages, including that $\CM$ need not be countable, and, in our case, that the generic path $P$ is first-order definable. We do not know whether an internal forcing is possible for $\RCAstar$.

Recall that $A$ is \emph{weak truth-table reducible} to $B$ (written $A \leq_{wtt} B$) if there is a computable function $f$ and a Turing reduction $A = \Phi^B_e$ such that $\Phi^B_{e,f(x)}(x) \halts$ for all $x$. A set $A$ is called \emph{low} if its jump satisfies $A' \leq_T 0'$, and \emph{superlow} if $A' \leq_{wtt} 0'$.
We are interested in superlow sets because of the following.
\begin{lemma}\label{lem:wtt}
 If $B\Sigma_1$ holds and $A \leq_{wtt} B$ and $B$ is regular, then $A$ is regular.
\end{lemma}
\begin{proof}
 Let $f$ be the computable bound. Fix any $x$ and let $s = \max_{y < x} f(y)$. Then $A \rest x = \{y < x : \Phi^B_{e,s} \halts = 1\}$ is a finite set.
\end{proof}
The Superlow Basis Theorem states that every infinite computable tree has an infinite superlow path. Its proof is implicit in that of the original Low Basis Theorem of Jockusch and Soare \cite{Jockusch-Soare:1972:CAD}.
\begin{theorem}[Superlow Basis Theorem]
 If $(M,\{A\}) \models P^- + \exp + \ISo$ and $T$ is an infinite $\Delta^A_1$ tree, then $T$ contains an infinite path $P$ such that $P' \leq_{wtt}T'$.
\end{theorem}
\begin{proof}[Proof (after Jockusch--Soare).]
 Let $(\varphi^{A \oplus X}_e)_{e\in M}$ enumerate all $\Sigma_1$ formulae taking $A$ as a parameter and with a free second-order variable $X$. Adopt the convention that, if $\sigma$ is a finite binary string and $S \subseteq M$ is the finite set whose characteristic function is $\sigma$, then $\varphi^{A \oplus \sigma}_e$ holds iff $\varphi^{A \oplus B}_e$ holds without looking at $B(x)$ for any $x \geq |\sigma|$. (As in the proof of Proposition \ref{prop:2.4}, we can justify this measure by finding a computable upper bound on the $\Delta_0$ matrix of the $\Sigma_1$ formula.)
 Now, for each $e$, let $U_e = \{\sigma \in 2^{<M} : \neg \varphi^{A \oplus \sigma}_e$ holds$\}$, and define a sequence of infinite binary trees by recursion:

 $\begin{array}{l}
  T_0 = T \\
  T_{2e+1} = \left\{\begin{array}{l}
                    T_{2e} \cap U_s~\mathrm{if~this~set~is~infinite} \\
                    T_{2e} ~\mathrm{otherwise}
             \end{array}\right. \\
  T_{2e+2} = T_{2e+1} \cap \{\tau : \tau \supseteq \sigma\},~\mathrm{where}~|\sigma|=e~\mathrm{and}~\sigma~\mathrm{is~leftmost~s.t.~this~set~is~infinite.}
 \end{array}$

 Then in the limit we are left with a unique path $P$, and for each $e$ we have $e \in P'$ iff we chose $T_{2e+1}=T_{2e}$ at stage $2e+1$. In other words, each initial segment $P' \rest e$ is the rightmost $\sigma$ of length $e$ such that $T \cap \bigcap \{U_s : \sigma(s)=1\}$ is infinite. We can determine what this is using no more than $2^n$ predetermined queries to $A'$. Hence $P' \leq_{wtt} A'$.
\end{proof}
This gives the conservativity we wanted.
\begin{corollary}
 Fix $n \geq 1$, let $(M,\{A\}) \models P^- + \exp + \ISo$, and let $T$ be an infinite $\Delta_1^A$ tree.
 \begin{enumerate}
  \item \label{cor:A.2:ii}If $(M,\{A\}) \models \ISn$, there is an infinite path $P$ such that $M\cup \{A \oplus P\} \models \ISn$.
  \item \label{cor:A.2:i} If $(M,\{A\}) \models \BSnp$, there is an infinite path $P$ such that $M\cup \{A \oplus P\} \models \BSn$.
 \end{enumerate}
\end{corollary}
\begin{proof}
  Let $P$ be as given by the Superlow Basis Theorem. The proofs for the various levels of induction are as follows:

  (Case $\ISo$.) Since $(M,\{A\})$ satisfies $\ISo$, the set $A'$ is regular. Since $A'$ is regular and $(A \oplus P)' \leq_{wtt} A'$, the set $(A \oplus P)'$ is regular by Lemma \ref{lem:wtt}. Since $(A \oplus P)'$ is regular, $(M,\{A \oplus P\})$ satisfies $\ISo$.

  (Case $\ISnp$.) Fix any $\Sigma_{n+1}^{A \oplus P}$ set $B$. Then $B$ is $\Sigma_n^{(A \oplus P)'}$. Then $B$ is $\Sigma_n^{A'}$. Then $B$ is $\Sigma_{n+1}^A$. Therefore $B$ is regular by $\ISnp$.

  (Case $\BSnp$.) Fix any $\Delta_{n+1}^{A \oplus P}$ set $B$. Then both $B$ and $\bar B$ are $\Sigma_{n+1}^A$ as in the previous case, hence $B$ is $\Delta^A_{n+1}$, hence $B$ is regular by $\BSnp$.
\end{proof}

\section{Strengthening jump inversion in a countable model}\label{app:B}
Towsner in \cite{towsner:2015:OMC} has recently shown the following.
\begin{theorem}\label{thm:B}
 If $\CM \models \RCA$ is countable and $C \subseteq M$ is any set, there is a set $B \subseteq M$ such that $\CM [B] \models \RCA$ and $C$ is $\Delta_2(B)$ (in fact, $\Delta_1(B')$).
\end{theorem}
 This differs from our own Lemma \ref{lem:2.6} (Friedberg jump inversion) in that it does not require $\CM$ to be topped or that $C$ satisfy any induction axioms; but it does require that $\CM$ be countable. In this section we give our own short proof of this theorem. We use Lemma \ref{lem:2.6} and a pair of new, if simple, lemmas: the first being an analogue of the Theorem but for $\Delta^0_1$ sets in place of $\Delta^0_2$, and the second showing that every countable model of $\RCA$ can be extended to a topped model.

 Both our proof and Towsner's use an external forcing and appear to require the countability of $\CM$. We do not know whether the assumption of countability can be relaxed.

 \begin{lemma}\label{lem:B.1}
  Let $\CM \models \RCAstar$ be countable and let $C \subseteq M$ be any set. Then there is a regular $D \subseteq M$ such that $C$ is $\Delta_1(D)$.
 \end{lemma}
 \begin{proof}
  Working outside of the model, let $(a_n)_{n \in \omega}$ be an enumeration of all elements of $M$. We let $D$ consist of triples of the form $\langle a_n, s_n, i_n\rangle$, where $i_n \in \{0,1\}$ with $i_n = 1$ iff $a \in C$, and the sequence $(\langle a_n,s_n,i_n\rangle)_{n \in \omega}$ is increasing and cofinal in $M$. Then $C$ is $\Delta_1(D)$, since $a \in C$ iff $(\exists s)\langle a,s,1\rangle \in D$ and $a \not \in C$ iff $(\exists s)\langle a,s,0\rangle \in D$. Furthermore, not only is each $D \rest n$ finite in the sense of Section \ref{sec:1.2}, but $D \rest n$ actually has cardinality $< \omega$. Hence $D$ is regular.
 \end{proof}
 For any regular set $A \subseteq M$, the join $A \oplus D$ is also regular. We prove Towsner's theorem first in a restricted form.
 \begin{proof}[Proof of Theorem \ref{thm:B}, when $\CM$ is topped.] Let $\CM, C$ be as in the theorem statement, and suppose $\CM$ is topped by $A$. Construct $D$ as in Lemma \ref{lem:B.1}. Then the join $A' \oplus D$ is regular. Apply Lemma \ref{lem:2.6} to obtain $B$ such that $\CM[B] \models \ISo$ and $D \leq_{wtt} (A \oplus B)'$ (not stated in the Lemma but implicit in its proof; $\leq_{wtt}$ is defined in Appendix \ref{app:A}), so that $C$ is $\Delta_1(B')$.
 \end{proof}
The more general case follows from this and the following.
\begin{lemma} \label{lem:B.2}
 If $\CM$ is a countable model of $\RCA$ then it can be extended to a topped countable model of $\RCA$.
\end{lemma}
\begin{proof}
 We imitate Spector's construction \cite{spector:56:OTD} of an exact pair for an ideal of Turing degrees, producing one set instead of two and using a type of Shore blocking to make sure it satisfies $\ISo$. Let $(A_n)_{n \in \omega}$ be an enumeration of all sets which are $\Delta^0_1$ in $\CM$, and let $(q_n)_{n \in \omega}$ be an increasing cofinal sequence in $M$ with $q_0 = 0$. It is enough to produce a set $B$ such that each $A_n$ is one-one reducible to $B$, and $\CM[B] \models \ISo$. If $B^{[y]}$ denotes the row $\{x : \langle x,y\rangle \in B\}$ and $=^*$ denotes equality modulo finite disagreement, our strategy will end with $B^{[q_n]} =^* A_n$ for all $q_n$ and $B^{[y]} =^* \emptyset$ for all other $y$, with a certain amount of forcing to ensure that $B'$ is regular. We build $B$ as a limit of partial approximations $(\beta_n)_{n \in \omega}$.

 Stage 0: Begin with $\beta_0 = \emptyset$.

 Stage $n+1$. Suppose the columns $\beta_n^{[y]}$ have been fully decided for all $y < q_n$, as well as some finite set of elements from other columns. Suppose also, by induction, that $\beta_n$ is $\Delta_1(A_0 \oplus \cdots \oplus A_{n-1})$. Let $\delta$ be an extension of $\beta_n$ (by finitely many elements) which maximizes $\{e < q_n : \Phi_e^\delta(e) \halts\}$; such a $\delta$ exists because $\ISo$ holds relative to $A_0 \oplus \cdots \oplus A_{n-1}$. Extend this $\delta$ to $\beta_{n+1}$ by setting the remaining elements of $\beta_{n+1}^{[q_n]}$ to equal those of $A_n$, and the remaining elements of $\beta^{[y]}_{n+1}$ to zero for each $y < q_{n+1}$. Proceed to the next stage.

 Now let $B$ be the limiting set $B=\bigcup_{n \in \omega} A_n$. Then each $A_n$ is coded into $B$, and both $B$ and its jump $B'$ are regular.
\end{proof}
By use of a more complicated forcing, Towsner is able to generalize the theorem, making a given $C$ become $\Delta^0_{n+1}$ while preserving $\RCA + \ISn$ for any fixed $n$. When we try to adapt the present proof to arbitrary $\ISn$, we encounter no trouble in Lemmas \ref{lem:2.6} and \ref{lem:B.1} but difficulties do arise in Lemma \ref{lem:B.2}, namely, if $\CM$ is nonstandard then $\CM[A_\omega]$ as constructed does not satsify $\ISt$, because the cofinal set $\{q_n\}_{n \in \omega}$ is $\Sigma_2(A)$ in general.  We would like to know whether Lemma \ref{lem:B.2} can be improved to work for all $\RCA + \ISn$.

\bibliography{./citations-coh}
\bibliographystyle{plain}

\end{document}